\documentclass[a4paper,11pt]{article}
\usepackage[a4paper, margin=1in]{geometry}
\usepackage[utf8]{inputenc}
\usepackage[T1]{fontenc}
\usepackage{graphicx}
\usepackage{amsmath, amssymb, amstext, mathtools}
\usepackage{amsthm}
\usepackage{pgfplots}
\usepackage{hyperref}

\usepackage{enumitem}
\setenumerate[1]{label=(\arabic*)}

\definecolor[named]{urlblue}{cmyk}{1,0.58,0,0.21}

\hypersetup{
 breaklinks=true,
 colorlinks=true,
 citecolor=purple!70!blue!60!black,
 linkcolor=purple!70!blue!90!black,
 urlcolor=urlblue,
 pdflang={en},
 pdftitle={Homomorphism-Distinguishing Closedness for Graphs of Bounded Tree-Width},
 pdfauthor={Daniel Neuen}
}

\newtheorem{theorem}{Theorem}[section]
\newtheorem{lemma}[theorem]{Lemma}

\newtheorem{conjecture}[theorem]{Conjecture}
\theoremstyle{definition}
\newtheorem{definition}[theorem]{Definition}
\theoremstyle{remark}

\newcommand{\CA}{{\mathcal A}}
\newcommand{\CB}{{\mathcal B}}

\newcommand{\CD}{{\mathcal D}}

\newcommand{\CF}{{\mathcal F}}

\newcommand{\CL}{{\mathcal L}}
\newcommand{\CM}{{\mathcal M}}

\newcommand{\CT}{{\mathcal T}}

\newcommand{\RR}{{\mathbb R}}

\DeclareMathOperator{\tw}{tw}
\DeclareMathOperator{\htw}{hdtw}

\DeclareMathOperator{\CFI}{CFI}
\DeclareMathOperator{\twCFI}{{CFI^{\textsf{x}}}}

\DeclareMathOperator{\CopRob}{CopRob}
\DeclareMathOperator{\BP}{BP}

\DeclareMathOperator{\sub}{sub}

\DeclareMathOperator{\spasm}{spasm}

\newcommand{\WL}[2]{\chi^{#1,#2}}
\newcommand{\WLit}[3]{\chi_{(#2)}^{#1,#3}}

\newcommand{\orcid}[1]{\href{https://orcid.org/#1}{\includegraphics[height=1.8ex]{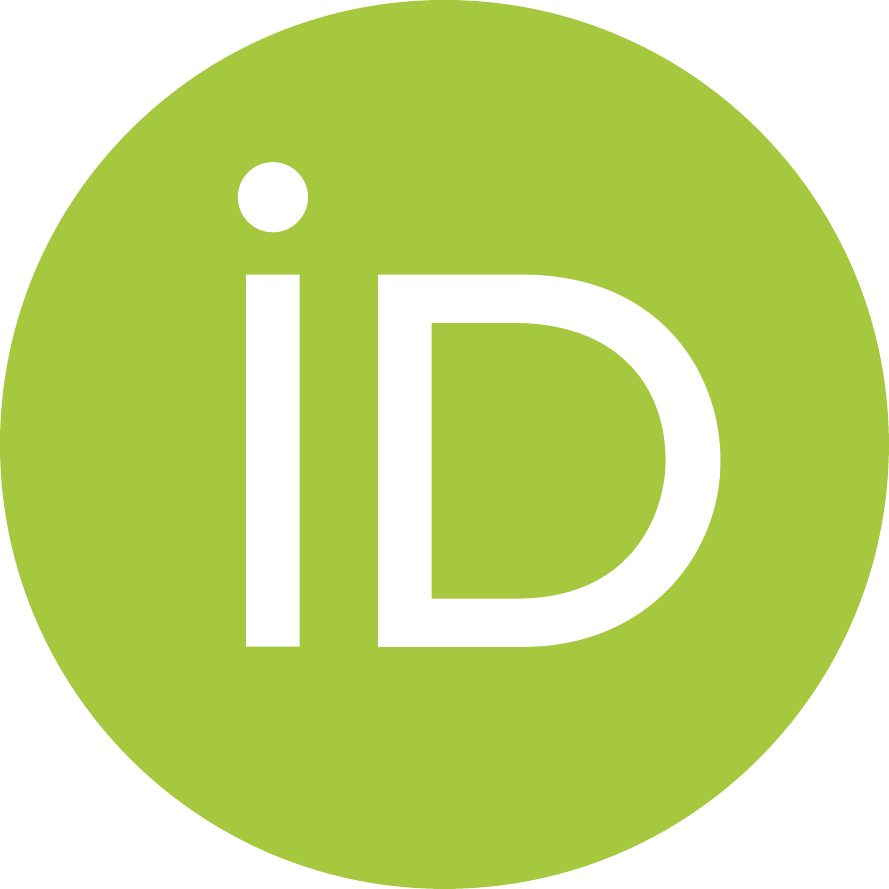}}}
\newcommand{\email}[1]{\href{mailto:#1}{\texttt{#1}}}

\title{Homomorphism-Distinguishing Closedness\\for Graphs of Bounded Tree-Width}
\author{
Daniel Neuen \orcid{0000-0002-4940-0318}\\
University of Bremen\\
\email{dneuen@uni-bremen.de}
}
\date{}

\begin{document}

\maketitle

\begin{abstract}
 Two graphs are \emph{homomorphism indistinguishable} over a graph class $\CF$, denoted by $G \equiv_\CF H$, if $\hom(F,G) = \hom(F,H)$ for all $F \in \CF$ where $\hom(F,G)$ denotes the number of homomorphisms from $F$ to $G$.
 A classical result of Lov\'{a}sz shows that isomorphism between graphs is equivalent to homomorphism indistinguishability over the class of all graphs.
 More recently, there has been a series of works giving natural algebraic and/or logical characterizations for homomorphism indistinguishability over certain restricted graph classes.

 A class of graphs $\CF$ is \emph{homomorphism-distinguishing closed} if, for every $F \notin \CF$, there are graphs $G$ and $H$ such that $G \equiv_\CF H$ and $\hom(F,G) \neq \hom(F,H)$.
 Roberson conjectured that every class closed under taking minors and disjoint unions is homomorphism-distinguishing closed which implies that every such class defines a distinct equivalence relation between graphs.
 In this work, we confirm this conjecture for the classes $\CT_k$, $k \geq 1$, containing all graphs of tree-width at most $k$.

 As an application of this result, we also characterize which subgraph counts are detected by the $k$-dimensional Weisfeiler-Leman algorithm.
 This answers an open question from [Arvind et al., J.\ Comput.\ Syst.\ Sci., 2020].
\end{abstract}

\section{Introduction}
\label{sec:intro}

In 1967, Lov\'{a}sz \cite{Lovasz67} proved that two graphs $G$ and $H$ are isomorphic if and only if $\hom(F,G) = \hom(F,H)$ for every graph $F$ where $\hom(F,G)$ denotes the number of homomorphisms from $F$ to $G$.
A natural follow-up question is to ask whether it is necessary to take the class of all graphs $F$ to obtain the above result, and which kind of other equivalence relations can be obtained  by restricting  $F$ to come from a proper subclass of all graphs.
For a graph class $\CF$, we say that two graphs $G$ and $H$ are \emph{$\CF$-equivalent}, denoted by $G \equiv_\CF H$, if $\hom(F,G) = \hom(F,H)$ for all $F \in \CF$.
Hence, Lov\'{a}sz's \cite{Lovasz67} result says that $\equiv_\CA$ is identical to the isomorphism relation where $\CA$ denotes the class of all graphs.

In recent years, there has been a series of works giving natural algebraic and/or logical characterizations for homomorphism indistinguishability over certain restricted classes of graphs.
For example, this includes graphs of bounded tree-width \cite{Dvorak10}, graphs of bounded path-width \cite{GroheRS22}, graphs of bounded tree-depth \cite{Grohe20,GroheRS22} and the class of planar graphs \cite{MancinskaR20}.
In particular, those results imply that the equivalence relations $\equiv_\CF$ obtained from the mentioned graph classes $\CF$ do not correspond to isomorphism, and moreover, these equivalence relations are pairwise distinct.

In \cite{Roberson22}, Roberson initiated a more systematic study of the question which types of graph classes $\CF$ lead to different equivalence relations $\equiv_\CF$.
A class of graphs $\CF$ is called \emph{homomorphism-distinguishing closed} if, for every $F \notin \CF$, there are graphs $G$ and $H$ such that $G \equiv_\CF H$ and $\hom(F,G) \neq \hom(F,H)$.

\begin{conjecture}[Roberson \cite{Roberson22}]
 \label{conj:hom-dist-closed}
 Let $\CF$ be a graph class closed under taking disjoint unions and minors. Then $\CF$ is homomorphism-distinguishing closed.
\end{conjecture}

In particular, this conjecture implies that every graph class closed under taking disjoint unions and minors defines a distinct equivalence relation $\equiv_\CF$.
Note that not every graph class is homomorphism-distinguishing closed.
For example, the class $\CD_2$ of $2$-degenerate graphs (which is not closed under taking minors) is not homomorphism-distinguishing closed since the corresponding equivalence relation defines the isomorphism relation between graphs \cite{Dvorak10}.

For $k \geq 1$ let $\CT_k$ denote the class of all graphs of tree-width at most $k$.
Roberson \cite{Roberson22} showed that $\CT_k$ is homomorphism-distinguishing closed for $k \in \{1,2\}$.
In this note, we generalize this to all $k \geq 1$.

\begin{theorem}
 \label{thm:tw-homomorphism-distinguishing-closed}
 The class $\CT_k$ is homomorphism-distinguishing closed for all $k \geq 1$.
\end{theorem}

For the proof, we rely on known characterizations of homomorphism indistinguishability over the class $\CT_k$ \cite{DellGR18,Dvorak10,GroheRS22} and existing constructions of non-isomorphic pairs of graphs that are difficult to distinguish (see, e.g., \cite{CaiFI92,DawarR07,Roberson22}).

\medskip

As an application of this result, we are able to characterize which subgraph counts are detected by the Weisfeiler-Leman algorithm (see also \cite{ArvindFKV20}).
The Weisfeiler-Leman algorithm (WL) is a standard heuristic in the context of graph isomorphism testing (see, e.g., \cite{CaiFI92}) which recently also gained attention in a machine learning context \cite{MorrisLMRKGFB21,MorrisRFHLRG19,ShervashidzeSLMB11,XuHLJ19}.
For $k \geq 1$, the $k$-dimensional Weisfeiler-Leman algorithm ($k$-WL) computes an isomorphism-invariant coloring of the $k$-tuples of vertices of a graph $G$.
If the color patterns computed for two graphs $G$ and $H$ do not match, the graphs are non-isomorphic.
In this case, we say that $k$-WL \emph{distinguishes} $G$ and $H$.
It is known that two graphs $G$ and $H$ are distinguished by $k$-WL if and only if $G \not\equiv_{\CT_k} H$, i.e., indistinguishability by $k$-WL can be characterized by homomorphism indistinguishability over the class of graphs of tree-width at most $k$ \cite{DellGR18,Dvorak10,GroheRS22}.

In \cite{Furer17}, F{\"{u}}rer initiated research on the question which subgraph counts are detected by $k$-WL.
Let $F$ and $G$ be two graphs. We write $\sub(F,G)$ to denote the number of subgraphs of $G$ isomorphic to $F$.
We say the function $\sub(F,\cdot)$ is \emph{$k$-WL invariant} if $\sub(F,G) = \sub(F,H)$ for all graphs $G,H$ that are indistinguishable by $k$-WL.
For example, F{\"{u}}rer \cite{Furer17} shows that $\sub(C_\ell,\cdot)$ is $2$-WL invariant for all $\ell \leq 6$ (where $C_\ell$ denotes the cycle on $\ell$ vertices), but $\sub(K_4,\cdot)$ is not $2$-WL invariant.
In \cite{ArvindFKV20}, Arvind, Fuhlbr{\"{u}}ck, K{\"{o}}bler and Verbitsky further extended this line of research by showing $\sub(F,\cdot)$ is $k$-WL invariant for all graphs $F$ that have hereditary tree-width at most $k$.
For a graph $F$ we define its \emph{hereditary tree-width}, denoted by $\htw(F)$, to be the maximum tree-width of a homomorphic image of $F$.
Arvind et al.\ \cite{ArvindFKV20} also provide some isolated negative results, but could not obtain a complete classification of which subgraph counts are detected by $k$-WL even for the special case $k=2$.

Building on Theorem \ref{thm:tw-homomorphism-distinguishing-closed}, we provide a complete classification which subgraph counts are detected by $k$-WL for all $k \geq 1$.
This answers an open question from \cite{ArvindFKV20}.

\begin{theorem}
 \label{thm:wl-subgraph-count}
 Let $F$ be a graph and $k \geq 1$.
 Then $\sub(F,\cdot)$ is $k$-WL invariant if and only if $\htw(F) \leq k$.
\end{theorem}

Observe that the backward direction is already proved in \cite{ArvindFKV20}, i.e., the main contribution of this work is to show that for every graph $F$ with $\htw(F) > k$, the $k$-WL algorithm fails to detect subgraph counts from $F$.

For the proof, we use a well-known result \cite{CurticapeanDM17} that allows us to formulate subgraph counts as a linear combination of certain homomorphism counts, and then combine Theorem \ref{thm:tw-homomorphism-distinguishing-closed} with an auxiliary lemma from \cite{Seppelt23}.

\section{Preliminaries}
\label{sec:prelim}

A \emph{graph} is a pair $G=(V,E)$ with vertex set $V = V(G)$ and edge relation $E = E(G)$.
In this paper all graphs are finite, simple (no loops or multiple edges), and undirected.
We denote edges by $vw \in E(G)$ where $v,w \in V(G)$.
The \emph{neighborhood} of~$v\in V(G)$ is denoted by~$N_G(v)$.
Moreover, we write~$E_G(v)$ to denote the set of edges incident to $v$.
If the graph is clear from context, we usually omit the index $G$ and simply write $N(v)$ and $E(v)$.
For $A \subseteq V(G)$ we denote by $G[A]$ the \emph{induced subgraph} of $G$ on $A$.
Also, we denote by $G\setminus A$ the induced subgraph on the complement of $A$, that is $G \setminus A \coloneqq G[V(G) \setminus A]$.

An \emph{isomorphism} from a graph $G$ to another graph $H$ is a bijective mapping $\varphi\colon V(G) \rightarrow V(H)$ which preserves the edge relation, that is, $vw \in E(G)$ if and only if $\varphi(v)\varphi(w) \in E(H)$ for all~$v,w \in V(G)$.
Two graphs $G$ and $H$ are \emph{isomorphic} ($G \cong H$) if there is an isomorphism from~$G$ to~$H$.
We write $\varphi\colon G\cong H$ to denote that $\varphi$ is an isomorphism from $G$ to $H$.

Let $F$ and $G$ be two graphs.
A \emph{homomorphism} from $F$ to $G$ is a mapping $\varphi\colon V(F) \rightarrow V(G)$ such that $\varphi(v)\varphi(w) \in E(G)$ for all $vw \in E(F)$.
We write $\hom(F,G)$ to denote the number of homomorphisms from $F$ to $G$.

Let $G$ be a graph.
A graph $H$ is a \emph{minor} of $G$ if $H$ can be obtained from $G$ by deleting vertices and edges of $G$ as well as contracting edges of $G$.
More formally, let $\CB = \{B_1,\dots,B_h\}$ be a partition of $V(G)$ such that $G[B_i]$ is connected for all $i \in [h]$.
We define $G/\CB$ to be the graph with vertex set $V(G/\CB) \coloneqq \CB$ and
\[E(G/\CB) \coloneqq \{BB' \mid \exists v \in B, v' \in B' \colon vv' \in E(G)\}.\]
A graph $H$ is a minor of $G$ if there is a partition $\CB = \{B_1,\dots,B_h\}$ of connected subsets $B_i \subseteq V(G)$ such that $H$ is isomorphic to a subgraph of $G/\CB$.
A graph $G$ \emph{excludes $H$ as a minor} if $H$ is not a minor of $G$.

\section{Homomorphism Indistinguishability and Oddomorphisms}

Toward the proof of Theorem \ref{thm:tw-homomorphism-distinguishing-closed}, we need to cover several tools introduced in \cite{Roberson22}.

\begin{definition}[Roberson \cite{Roberson22}]
 Let $F$ and $G$ be graphs and suppose $\varphi$ is a homomorphism from $F$ to $G$.
 We say a vertex $a \in V(F)$ is \emph{odd (with respect to $\varphi$)} if $|N_F(a) \cap \varphi^{-1}(v)|$ if odd for every $v \in N_G(\varphi(a))$.
 Similarly, we say a vertex $a \in V(F)$ is \emph{even with respect to $\varphi$} if $|N_F(a) \cap \varphi^{-1}(v)|$ if even for every $v \in N_G(\varphi(a))$.
 
 An \emph{oddomorphism} from $F$ to $G$ is a homomorphism $\varphi$ from $F$ to $G$ such that
 \begin{enumerate}[label = (\Roman*)]
  \item every vertex $a \in V(F)$ is odd or even (with respect to $\varphi$), and
  \item $\varphi^{-1}(v)$ contains an odd number of odd vertices for every $v \in V(G)$.
 \end{enumerate}
 A \emph{weak oddomorphism} from $F$ to $G$ is a homomorphism $\varphi$ from $F$ to $G$ such that there is a subgraph $F'$ of $F$ for which $\varphi|_{V(F')}$ is an oddomorphism from $F'$ to $G$.
\end{definition}

Next, we introduce a construction for pairs of similar graphs from a base graph $G$ that has also been used in \cite{Roberson22}.
Actually, variants of this construction have already been used in several earlier works (see, e.g.,\ \cite{CaiFI92,DawarR07}).

Let $G$ be a graph and let $U \subseteq V(G)$.
For $v \in V(G)$ we define $\delta_{v,U} \coloneqq |\{v\} \cap U|$.
We define the graph $\CFI(G,U)$ (the name refers to the authors of \cite{CaiFI92} where this construction was first used in a related context) with vertex set
\[V(\CFI(G,U)) \coloneqq \{(v,S) \mid v \in V(G), S \subseteq E(v), |S| \equiv \delta_{v,U} \bmod 2\}\]
and edge set
\[V(\CFI(G,U)) \coloneqq \{(v,S)(u,T) \mid uv \in E(G), uv \notin S \bigtriangleup T\}\]
(here, $S \bigtriangleup T$ denotes the symmetric difference of $S$ and $T$, i.e., $S \bigtriangleup T \coloneqq (S \setminus T) \cup (T \setminus S)$).
For $v \in V(G)$ we also write $M_{G,U}(v) \coloneqq \{(v,S) \mid S \subseteq E(v), |S| \equiv \delta_{v,U} \bmod 2\}$ for the vertices in $\CFI(G,U)$ associated with $v$.

The following lemma is well-known (see, e.g., \cite{CaiFI92,Roberson22})

\begin{lemma}
 Let $G$ be a connected graph and let $U,U' \subseteq V(G)$.
 Then $\CFI(G,U) \cong \CFI(G,U')$ if and only if $|U| \equiv |U'| \bmod 2$.
\end{lemma}

We define $\CFI(G) \coloneqq \CFI(G,\emptyset)$ and $\twCFI(G) \coloneqq \CFI(G,\{u\})$ for some $u \in V(G)$.

\begin{theorem}[Roberson {\cite[Theorem 3.13]{Roberson22}}]
 \label{thm:hom-count-cfi-oddomorphism}
 Let $F,G$ be graphs and suppose $G$ is connected.
 Then $\hom(F,\CFI(G)) \geq \hom(F,\twCFI(G))$.
 Moreover, $\hom(F,\CFI(G)) > \hom(F,\twCFI(G))$ if and only if there exists a weak oddomorphism from $F$ to $G$.
\end{theorem}

We require two additional tools from \cite{Roberson22} stated below.

\begin{lemma}[{\cite[Lemma 5.6]{Roberson22}}]
 \label{la:oddomorphism-minors}
 Let $F$ and $G$ be graphs such that there is a weak oddomorphism from $F$ to $G$.
 Also suppose $G'$ is a minor of $G$.
 Then there is a minor $F'$ of $F$ such that there is an oddomorphism from $F'$ to $G'$.
\end{lemma}

\begin{lemma}[{\cite[Theorem 6.2]{Roberson22}}]
 \label{la:homomorphism-distinguishing-closed-condition}
 Let $\CF$ be a class of graphs such that
 \begin{enumerate}
  \item \label{item:homomorphism-distinguishing-closed-condition-1} if $F \in \CF$ and there is a weak oddomorphism from $F$ to $G$, then $G \in \CF$, and
  \item \label{item:homomorphism-distinguishing-closed-condition-2} $\CF$ is closed under disjoint unions and restrictions to connected components.
 \end{enumerate}
 Then $\CF$ is homomorphism-distinguishing closed.
\end{lemma}

\section{Graphs of Bounded Tree-Width}

In this section, we present the proof of Theorem \ref{thm:tw-homomorphism-distinguishing-closed}.
We rely on game characterizations for graphs of bounded tree-width as well as homomorphism indistinguishability over graphs of tree-width at most $k$.

\subsection{Games}

First, we cover the cops-and-robber game that characterizes tree-width of graphs.
Fix some integer $k \geq 1$.
For a graph $G$, we define the cops-and-robber game $\CopRob_k(G)$ as follows:

\begin{itemize}
 \item The game has two players called \emph{Cops} and \emph{Robber}.
 \item The game proceeds in rounds, each of which is associated with a pair of positions
 $(\bar v,u)$ with~$\bar v \in \big(V(G)\big)^k$ and~$u \in V(G)$.
 \item To determine the initial position, the Cops first choose a tuple $\bar v = (v_1,\dots,v_k) \in \big(V(G)\big)^k$ and then the Robber chooses some vertex $u \in V(G) \setminus \{v_1,\dots,v_k\}$ (if no such $u$ exists, the Cops win the play).
  The initial position of the game is then set to $(\bar v,u)$.
 \item Each round consists of the following steps.
  Suppose the current position of the game is $(\bar v,u) = ((v_1,\dots,v_k),u)$.
  \begin{itemize}
   \item[(C)] The Cops choose some $i \in [k]$ and $v' \in V(G)$.
   \item[(R)] The Robber chooses a vertex $u' \in V(G)$ such that there exists a path from $u$ to $u'$ in $G \setminus \{v_1,\dots,v_{i-1},v_{i+1},\dots,v_k\}$.
    After that, the game moves to position $\big((v_1,\dots,v_{i-1},v',v_{i+1},\dots,v_k), u'\big)$.
  \end{itemize}

  If $u \in \{v_1,\dots,v_k\}$ the Cops win.
  If there is no position of the play such that the Cops win, then the Robber wins.
\end{itemize}

We say that the Cops (and the Robber, respectively) \emph{win $\CopRob_k(G)$} if the Cops (and the Robber, respectively) have a winning strategy for the game.
We also say that \emph{$k$ cops can catch a robber on $G$} if the Cops have a winning strategy in this game.

\begin{theorem}[\cite{SeymourT93}]
 \label{thm:cops-and-robbers-characterization-of-treewidth}
 A graph $G$ has tree-width at most $k$ if and only if $k+1$ cops can catch a robber on $G$.
\end{theorem}

Next, we discuss a game-theoretic characterization of two graphs being indistinguishable via homomorphism counts from graphs of tree-width at most $k$.

Let $k \geq 1$.
For graphs $G$ and $H$ on the same number of vertices, we define the \emph{bijective $k$-pebble game} $\BP_{k}(G,H)$ as follows:
\begin{itemize}
 \item The game has two players called \emph{Spoiler} and \emph{Duplicator}.
 \item The game proceeds in rounds, each of which is associated with a pair of positions
 $(\bar v,\bar w)$ with~$\bar v \in \big(V(G)\big)^k$ and~$\bar w \in \big(V(H)\big)^k$.
 \item To determine the initial position, Duplicator plays a bijection $f\colon \big(V(G)\big)^k \rightarrow \big(V(H)\big)^k$ and Spoiler chooses some $\bar v \in \big(V(G)\big)^k$.
  The initial position of the game is then set to $(\bar v,f(\bar v))$.
 \item Each round consists of the following steps.
  Suppose the current position of the game is $(\bar v,\bar w) = ((v_1,\ldots,v_k),(w_1,\ldots,w_k))$.
  \begin{itemize}
   \item[(S)] Spoiler chooses some $i \in [k]$.
   \item[(D)] Duplicator picks a bijection $f\colon V(G) \rightarrow V(H)$.
   \item[(S)] Spoiler chooses $v \in V(G)$ and sets $w \coloneqq f(v)$.
    Then the game moves to position $\big(\bar v[i/v], \bar v[i/v]\big)$ where $\bar v[i/v] \coloneqq (v_1,\dots,v_{i-1},v,v_{i+1},\dots,v_k)$ is the tuple obtained from $\bar v$ by replacing the $i$-th entry by $v$.
  \end{itemize}

  If mapping each $v_i$ to $w_i$ does not define an isomorphism of the induced subgraphs of $G$ and $H$, Spoiler wins the play.
  More precisely, Spoiler wins if there are~$i,j\in [k]$ such that~$v_i = v_j \nLeftrightarrow w_i =w_j$ or~$v_iv_j \in E(G) \nLeftrightarrow w_iw_j \in E(H)$.
  If there is no position of the play such that Spoiler wins, then Duplicator wins.
\end{itemize}

We say that Spoiler (and Duplicator, respectively) \emph{wins $\BP_k(G,H)$} if Spoiler (and Duplicator, respectively) has a winning strategy for the game.
Also, for a position $(\bar v,\bar w)$ with $\bar v \in \big(V(G)\big)^k$ and $\bar w \in \big(V(H)\big)^k$, we say that Spoiler (and Duplicator, respectively) \emph{wins $\BP_k(G,H)$ from position $(\bar v,\bar w)$} if Spoiler (and Duplicator, respectively) has a winning strategy for the game started at position $(\bar v,\bar w)$.

The following theorem follows from \cite{CaiFI92} and \cite{DellGR18,Dvorak10,GroheRS22}.

\begin{theorem}
 \label{thm:eq-hom-pebble}
 Suppose $k \geq 1$.
 Let $G$ and $H$ be two graphs.
 Then $\hom(F,G) = \hom(F,H)$ for every $F \in \CT_k$ if and only if Duplicator wins the game $\BP_{k+1}(G,H)$.
\end{theorem}

\subsection{Indistinguishable Graphs}

The main step in the proof of Theorem \ref{thm:tw-homomorphism-distinguishing-closed} is to show that $\CFI(G)$ and $\twCFI(G)$ can not be distinguished via homomorphism counts from graphs of tree-width at most $k$ for all connected graphs $G$ of tree-width strictly greater than $k$.
The proof follows similar arguments from \cite{DawarR07} used to prove a closely related statement.
Toward this end, the next lemma provides certain useful isomorphisms between CFI-graphs.

\begin{lemma}
 \label{la:twist-along-path}
 Let $G$ be a connected graph and suppose $u,v \in V(G)$.
 Let $P$ be a path from $u$ to $v$.
 Then there is an isomorphism $\varphi\colon \CFI(G,\{u\}) \cong \CFI(G,\{v\})$ such that
 \begin{enumerate}
  \item \label{item:twist-along-path-1} $\varphi(M_{G,\{u\}}(w)) = M_{G,\{v\}}(w)$ for all $w \in V(G)$, and
  \item \label{item:twist-along-path-2} $\varphi(w,S) = (w,S)$ for all $w \in V(G) \setminus V(P)$ and $(w,S) \in M_{G,\{u\}}(w)$.
 \end{enumerate}
\end{lemma}

\begin{proof}
 Let $E(P)$ denote the set of edges on the path $P$.
 Clearly,
 \begin{itemize}
  \item $|E(P) \cap E(u)| = 1$ and $|E(P) \cap E(v)| = 1$,
  \item $|E(P) \cap E(w)| = 2$ for all $w \in V(P) \setminus \{u,v\}$, and
  \item $|E(P) \cap E(w)| = 0$ for all $w \in V(G) \setminus V(P)$.
 \end{itemize}
 We define $\varphi(w,S) \coloneqq (w,S \bigtriangleup (E(P) \cap E(w)))$ for all $(w,S) \in \CFI(G,\{u\})$.
 It is easy to check that $\varphi\colon \CFI(G,\{u\}) \cong \CFI(G,\{v\})$ and the desired properties are satisfied.
\end{proof}

The next lemma forms the key technical step in the proof of Theorem \ref{thm:tw-homomorphism-distinguishing-closed}.

\begin{lemma}
 \label{la:cfi-duplicator-win}
 Let $G$ be a connected graph of tree-width $\tw(G) \geq k$.
 Then Duplicator wins the $k$-bijective pebble game played on $\CFI(G)$ and $\twCFI(G)$.
\end{lemma}

\begin{proof}
 Let us fix some vertex $u_0 \in V(G)$ so that $\twCFI(G) = \CFI(G,\{u_0\})$.
 Since $\tw(G) \geq k$, the Robber has a winning strategy in the cops-and-robber game $\CopRob_k(G)$ by Theorem \ref{thm:cops-and-robbers-characterization-of-treewidth}.
 We translate the winning strategy for the Robber in $\CopRob_k(G)$ into a winning strategy for Duplicator in the $k$-bijective pebble game played on $\CFI(G)$ and $\twCFI(G)$.

 We first construct the bijection $f$ for the initialization round.
 Suppose $\bar x = (x_1,\dots,x_k) \in (V(\CFI(G)))^k$.
 We define $A(\bar x) \coloneqq (v_1,\dots,v_k)$ where $v_i \in V(G)$ is the unique vertex such that $x_i \in M_{G,\emptyset}(v_i)$.

 Now let $u$ be the vertex chosen by the Robber if the Cops initially place themselves on $A(\bar x)$.
 Let $P$ be a shortest path from $u$ to $u_0$ (recall that $G$ is connected), and let $\varphi$ denote the isomorphism from $\CFI(G,\{u\})$ to $\CFI(G,\{u_0\})$ constructed in Lemma \ref{la:twist-along-path}.
 We set $f(\bar x) \coloneqq (\varphi(x_1),\dots,\varphi(x_k))$.
 It is easy to see that this gives a bijection $f$ (we use the same isomorphism $\varphi$ for all tuples $\bar x$ having the same associated tuple $A(\bar x)$).

 Now, throughout the game, Duplicator maintains the following invariant.
 Let $(\bar x,\bar y)$ denote the current position.
 Then there is a vertex $u \in V(G)$ and a bijection $\varphi\colon \CFI(G,\{u\}) \cong \CFI(G,\{u_0\})$ such that
 \begin{itemize}
  \item $\varphi(M_{G,\{u\}}(w)) = M_{G,\{u_0\}}(w)$ for all $w \in V(G)$,
  \item $\varphi(\bar x) = \bar y$,
  \item $u$ does not appear in the tuple $A(\bar x)$, and
  \item the Robber wins from the position $(A(\bar x), u)$, i.e., if the Cops are placed on $A(\bar x)$ and the Robber is on $u$.
 \end{itemize}
 Note that this condition is satisfied by construction after the initialization round.

 Also observe that Duplicator never looses the game in such a position.
 Indeed, the mapping $\varphi$ restricts to an isomorphism between $\CFI(G,\emptyset) - M_{G,\emptyset}(u) = \CFI(G,\{u\}) - M_{G,\{u\}}(u)$ and $\CFI(G,\{u_0\}) - M_{G,\{u_0\}}(u)$.
 Hence, since no vertex associated with $u$ is pebbled in either graph, the pair $(\bar x,\bar y)$ induces a local isomorphism.

 So it remains to show that Duplicator can maintain the above invariant in each round of the $k$-bijective pebble game.
 Suppose $(\bar x,\bar y)$ is the current position.
 Also let $(A(\bar x), u)$ be the associated position in the cops-and-robber game.
 Suppose that $A(\bar x) = (v_1,\dots,v_k)$.

 Let $i \in [k]$ denote the index chosen by Spoiler.
 We describe the bijection $f$ chosen by Duplicator.
 Let $v \in V(G)$.
 Let $u'$ be the vertex the Robber moves to if the Cops choose $i$ and $v$ (i.e., the $i$-th cop changes its position to $v$) in the position $(A(\bar x),u)$.
 Let $P$ denote a path from $u$ to $u'$ that avoids $\{v_1,\dots,v_k\} \setminus \{v_i\}$.
 Let $\psi$ denote the isomorphism from $\CFI(G,\{u'\})$ to $\CFI(G,\{u\})$ constructed in Lemma \ref{la:twist-along-path}.
 We set $f(x) \coloneqq \varphi(\psi(x))$ for all $x \in M_{G,\emptyset}(v)$.

 It is easy to see that $f$ is a bijection.
 Let $x$ denote the vertex chosen by Spoiler and let $y \coloneqq f(x)$.
 Let $\bar x' \coloneqq \bar x[i/x]$ and $\bar y' \coloneqq \bar y[i/y]$, i.e., the pair $(\bar x',\bar y')$ is the new position of the game.
 Also, we set $\varphi' \coloneqq \psi \circ \varphi$ where $\psi$ denotes the isomorphism from $\CFI(G,\{u'\})$ to $\CFI(G,\{u\})$ used in the definition of $f(x)$.

 Clearly, $\varphi'(M_{G,\{u'\}}(w)) = M_{G,\{u_0\}}(w)$ for all $w \in V(G)$, since the corresponding conditions are satisfied for the mappings $\psi$ and $\varphi$.
 We have $\varphi'(x) = y$ by definition.
 All the other entries of $\bar x'$ are fixed by the mapping $\psi$ (see Lemma \ref{la:twist-along-path}\ref{item:twist-along-path-2}) which overall implies that $\varphi'(\bar x') = \bar y'$.
 Also, $u'$ does not appear in the tuple $A(\bar x')$ by construction, and the Robber wins from the position $(A(\bar x'), u')$.

 So overall, this means that Duplicator can maintain the above invariant which provides the desired winning strategy.
\end{proof}

With this, we are now ready to prove Theorem \ref{thm:tw-homomorphism-distinguishing-closed}.

\begin{proof}[Proof of Theorem \ref{thm:tw-homomorphism-distinguishing-closed}]
 Let $k \geq 1$ be fixed.
 Clearly, the class $\CT_k$ satisfies Condition \ref{item:homomorphism-distinguishing-closed-condition-2} from Lemma \ref{la:homomorphism-distinguishing-closed-condition}.

 Hence, consider some $F \in \CT_k$ and suppose there is a weak oddomorphism from $F$ to $G$.
 Suppose towards a contradiction that $G \notin \CT_k$, i.e., $\tw(G) > k$.
 Then there is a connected subgraph $G'$ of $G$ such that $\tw(G') > k$.
 By Lemma \ref{la:oddomorphism-minors}, we conclude that there is a graph $F' \in \CT_k$ such that there is an oddomorphism from $F'$ to $G'$.
 By Theorem \ref{thm:hom-count-cfi-oddomorphism}, we conclude that $\hom(F',\CFI(G')) > \hom(F',\twCFI(G'))$.
 Using Theorem \ref{thm:eq-hom-pebble} it follows that Spoiler wins the $(k+1)$-bijective pebble game $\BP_{k+1}(\CFI(G'),\twCFI(G'))$.
 But this contradicts Lemma \ref{la:cfi-duplicator-win} since $\tw(G') \geq k+1$.

 So $G \in \CT_k$ which means that the class $\CT_k$ also satisfies Condition \ref{item:homomorphism-distinguishing-closed-condition-1} from Lemma \ref{la:homomorphism-distinguishing-closed-condition}.
 Hence, $\CT_k$ is homomorphism-distinguishing closed by Lemma \ref{la:homomorphism-distinguishing-closed-condition}
\end{proof}

\section{Weisfeiler-Leman and Subgraph Counts}

In this section, we prove Theorem \ref{thm:wl-subgraph-count}.
Towards this end, we first need to formally introduce the WL algorithm.

\subsection{The Weisfeiler-Leman Algorihtm}

Let $\chi_1,\chi_2\colon V^k \rightarrow C$ be colorings of the $k$-tuples of vertices, where $C$ is some finite set of colors.
We say $\chi_1$ \emph{refines} $\chi_2$, denoted $\chi_1 \preceq \chi_2$, if $\chi_1(\bar v) = \chi_1(\bar w)$ implies $\chi_2(\bar v) = \chi_2(\bar w)$ for all $\bar v,\bar w \in V^{k}$.
The colorings $\chi_1$ and $\chi_2$ are \emph{equivalent}, denoted $\chi_1 \equiv \chi_2$,  if $\chi_1 \preceq \chi_2$ and $\chi_2 \preceq \chi_1$.

We describe the \emph{$k$-dimensional Weisfeiler-Leman algorithm} ($k$-WL) for all $k \geq 1$.
For an input graph $G$ let $\WLit{k}{0}{G}\colon (V(G))^{k} \rightarrow C$ be the coloring where each tuple is colored with the isomorphism type of its underlying ordered subgraph.
More precisely, $\WLit{k}{0}{G}(v_1,\dots,v_k) = \WLit{k}{0}{G}(v_1',\dots,v_k')$ if and only if, for all $i,j \in [k]$, it holds that
$v_i = v_j \Leftrightarrow v_i'= v_j'$ and $v_iv_j \in E(G) \Leftrightarrow v_i'v_j' \in E(G)$.

We then recursively define the coloring $\WLit{k}{i+1}{G}$ obtained after $i+1$ rounds of the algorithm (for $i \geq 0$).
For $k \geq 2 $ and $\bar v = (v_1,\dots,v_k) \in (V(G))^k$ we define
\[\WLit{k}{i+1}{G}(\bar v) \coloneqq \Big(\WLit{k}{i}{G}(\bar v), \CM_i(\bar v)\Big)\]
where
\[\CM_i(\bar v) \coloneqq \Big\{\!\!\Big\{\big(\WLit{k}{i}{G}(\bar v[w/1]),\dots,\WLit{k}{i}{G}(\bar v[w/k])\big) \;\Big\vert\; w \in V(G) \Big\}\!\!\Big\}\]
and $\bar v[w/i] \coloneqq (v_1,\dots,v_{i-1},w,v_{i+1},\dots,v_k)$ is the tuple obtained from $\bar v$ by replacing the $i$-th entry by $w$.
For $k = 1$, the definition is similar, but we only iterate over neighbors of $v_1$, i.e.,
\[\CM_i(v_1) \coloneqq \Big\{\!\!\Big\{\WLit{k}{i}{G}(w) \;\Big\vert\; w \in N_G(v_1) \Big\}\!\!\Big\}.\]
There is a minimal $i_\infty \geq 0$ such that $\WLit{k}{i_{\infty}}{G} \equiv \WLit{k}{i_{\infty}+1}{G}$ and for this $i_\infty$ we define $\WL{k}{G} \coloneqq \WLit{k}{i_\infty}{G}$.

Let $G$ and $H$ be two graphs.
We say that $k$-WL \emph{distinguishes} $G$ and $H$ if there exists a color $c$ such that
\[\Big|\Big\{\bar v \in \big(V(G)\big)^{k} \;\Big\vert\; \WL{k}{G}(\bar v) = c \Big\}\Big| \neq \Big|\Big\{\bar w \in \big(V(H)\big)^{k} \;\Big\vert\; \WL{k}{H}(\bar w) = c\Big\}\Big|.\]
We write $G \simeq_k H$ if $k$-WL does not distinguish $G$ and $H$.

Recall that $\CT_k$ denotes the class of graphs of tree-width at most $k$.
The following characterization follows from \cite{DellGR18,Dvorak10,GroheRS22} (see also Theorem \ref{thm:eq-hom-pebble}).

\begin{theorem}
 \label{thm:eq-hom-wl}
 Suppose $k \geq 1$.
 Let $G$ and $H$ be two graphs.
 Then $G \simeq_k H$ if and only if $G \equiv_{\CT_k} H$.
\end{theorem}

Recall that we write $\sub(F,G)$ to denote the number of subgraphs of $G$ isomorphic to $F$.
We write $\sub(F,\cdot)$ to denote the function that maps each graph $G$ to the corresponding subgraph count $\sub(F,G)$.

\begin{definition}
 Let $F$ be a graph.
 The function $\sub(F,\cdot)$ is \emph{$k$-WL invariant} if
 \begin{equation}
  \sub(F,G) = \sub(F,H)
 \end{equation}
 for all graphs $G,H$ such that $G \simeq_k H$.
\end{definition}

\subsection{Subgraph Counts}

Using the framework from \cite{CurticapeanDM17}, it is possible to describe the subgraph count $\sub(F,G)$ as a linear combination
\[\sub(F,G) = \sum_{i \in [\ell]} \alpha_i \cdot \hom(F_i,G)\]
for certain graphs $F_1,\dots,F_\ell$ and coefficients $\alpha_1,\dots,\alpha_\ell \in \RR$ that only depend on $F$.
More precisely, the graphs $F_1,\dots,F_\ell$ are exactly the homomorphic images of $F$.

\begin{definition}
 Let $F$ and $H$ be two graphs.
 We say that $H$ is a \emph{homomorphic image} of $F$ if there is a surjective homomorphism $\varphi\colon V(F) \to V(H)$
 such that
 \[E(H) = \{\varphi(v)\varphi(w) \mid vw \in E(F)\}.\]
 We write $\spasm(F)$ to denote the set of homomorphic images of $F$.
 The \emph{hereditary tree-width} of $F$, denoted by $\htw(F)$, is the maximum tree-width of a graph in $\spasm(F)$, i.e.,
 \[\htw(F) \coloneqq \max_{H \in \spasm(F)} \tw(H).\]
\end{definition}

In the following, we assume that $\spasm(F)$ contains only one representative from each isomorphism class, i.e., for every homomorphic image $H$ of $F$ there is exactly one graph $H' \in \spasm(F)$ that is isomorphic to $H$.
In particular, the set $\spasm(F)$ is finite.

The backward direction of Theorem \ref{thm:wl-subgraph-count} has already been proved in \cite{ArvindFKV20}.

\begin{lemma}[{\cite[Corollary 4.3]{ArvindFKV20}}]
 \label{la:wl-subgraph-count-backward}
 Let $F$ be a graph such that $\htw(F) \leq k$.
 Then $\sub(F,\cdot)$ is $k$-WL invariant.
\end{lemma}

For the other direction, we combine Theorem \ref{thm:tw-homomorphism-distinguishing-closed} and the following lemma from \cite{Seppelt23}.

\begin{lemma}[{\cite[Lemma 4]{Seppelt23}}]
 \label{la:hom-count-to-closure}
 Let $\CF$ be a class of graphs that is homomorphism-distinguishing closed.
 Let $\CL$ be a finite set of pairwise non-isomorphic graphs and $\alpha\colon \CL \rightarrow \RR\setminus \{0\}$.
 Also suppose that for all graphs $G,H$ it holds that
 \begin{equation}
  \label{eq:hom-count-condition}
  G \equiv_\CF H \quad\implies\quad \sum_{L \in \CL} \alpha(L) \cdot \hom(L,G) = \sum_{L \in \CL} \alpha(L) \cdot \hom(L,H).
 \end{equation}
 Then $\CL \subseteq \CF$.
\end{lemma}

\begin{lemma}
 \label{la:wl-subgraph-count-forward}
 Let $F$ be a graph such that $\sub(F,\cdot)$ is $k$-WL invariant.
 Then $\htw(F) \leq k$.
\end{lemma}

\begin{proof}
 Let $\CF$ denote the class of graphs of tree-width at most $k$.
 By Theorem \ref{thm:tw-homomorphism-distinguishing-closed} the class $\CF$ is homomorphism-distinguishing closed.
 Let $\CL \coloneqq \spasm(F)$.
 By \cite{CurticapeanDM17} there is a unique function $\alpha\colon \CL \rightarrow \RR\setminus \{0\}$ such that
 \[\sub(F,G) = \sum_{L \in \CL} \alpha(L) \cdot \hom(L,G)\]
 for all graphs $G$.
 Since $\sub(F,\cdot)$ is $k$-WL invariant it follows that Equation \eqref{eq:hom-count-condition} is satisfied for all graphs $G,H$ using Theorem \ref{thm:eq-hom-wl}.
 So $\spasm(F) = \CL \subseteq \CF$ by Lemma \ref{la:hom-count-to-closure}.
 This implies that $\htw(F) \leq k$.
\end{proof}

\begin{proof}[Proof of Theorem \ref{thm:wl-subgraph-count}]
 The theorem follows directly from Lemmas \ref{la:wl-subgraph-count-backward} and \ref{la:wl-subgraph-count-forward}.
\end{proof}

\bibliographystyle{plainurl}
\small
\bibliography{literature}

\end{document}